\documentclass[12pt]{amsart}
\usepackage{amsfonts, amstext}
\usepackage[usenames,dvipsnames]{xcolor}	
\usepackage{amsmath, amsthm, amssymb, amsfonts}						
\usepackage{fullpage}
\usepackage{enumitem}
\usepackage{mathrsfs}
\usepackage{tikz-cd}
\usepackage{hhline}
\usepackage{textcomp}
\usepackage{stmaryrd}

\usepackage{hyperref}
\hypersetup{colorlinks, linkcolor={black}, citecolor={black}}

\newcommand{\Z}{\mathbb{Z}}
\newcommand{\mf}{\mathfrak}

\newcommand{\bb}[1]{\mathbb{#1}}								

\newcommand{\map}[1]{\mathop{\longrightarrow}\limits^{#1}}
\newcommand{\Q}{\mathbb{Q}}

\newcommand{\A}{\mathbb{A}}

\newcommand{\R}{\mathbb{R}}
\newcommand{\mcal}{\mathcal}

\newcommand{\mscr}{\mathscr}
\newcommand*{\sHom}{\mscr{H}\kern -.5pt om}
\newcommand*{\sExt}{\mscr{E}\kern -.5pt xt}

\newcommand{\ov}{\overline}
\renewcommand{\div}{\mathrm{div}}
\newcommand{\HK}{e_{\mathrm{HK}}}
\newcommand{\nvol}{\widehat{\mathrm{vol}}}

\renewcommand{\O}{\mcal{O}}

\DeclareMathOperator{\Spec}{Spec}
\DeclareMathOperator{\Frac}{Frac}

\DeclareMathOperator{\Hom}{Hom}
\DeclareMathOperator{\Supp}{Supp}

\DeclareMathOperator{\Ram}{Ram}

\DeclareMathOperator{\dv}{div}

\DeclareMathOperator{\Tr}{Tr}

\DeclareMathOperator{\Diff}{Diff}
\DeclareMathOperator{\FDiff}{F-Diff}

\newtheorem{theorem}{Theorem}[section]
\newtheorem{proposition}[theorem]{Proposition}
\newtheorem{corollary}[theorem]{Corollary}
\newtheorem{lemma}[theorem]{Lemma}
\newtheorem{question}[theorem]{Question}

\theoremstyle{definition}
\newtheorem{example}[theorem]{Example}
\newtheorem{definition}[theorem]{Definition}

\theoremstyle{remark}
\newtheorem{remark}[theorem]{Remark}

\title{Inversion of adjunction for $F$-signature}
\author{Gregory Taylor}
\date{}

\begin{document}

\maketitle

\begin{abstract}
Let $(R,\Delta+D)$ be a log $\mathbb{Q}$-Gorenstein pair where $R$ is a Noetherian, $F$-finite, normal, local domain of characteristic $p > 0$, $\Delta$ is an effective $\mathbb{Q}$-divisor and $D$ is an integral $\Q$-Cartier divisor. We show that the left derivative of the $F$-signature function $s(R,\Delta + tD)$ at $t = 1$ is equal to $-s(\O_D, \Diff_D(\Delta))$. This equality is interpreted as a quantitative form of inversion of adjunction for strong $F$-regularity. As an immediate corollary, we obtain the inequality $s(R,\Delta) \geq s(\O_D, \mathrm{Diff}_D(\Delta))$. We also discuss the implications of our result for the conjectured connection between the $F$-signature and the normalized volume.
\end{abstract}

\section{Introduction}\label{sec:Introduction}

A local property $\mcal{P}$ satisfies inversion of adjunction if the implication 
\[
	D \text{ has } \mcal{P} \text{ for some prime divisor } D \subseteq X ~~~ \Longrightarrow ~~~ X \text{ has } \mcal{P} \text{ near } D
\] 
holds. In the minimal model program (MMP), inversion of adjunction allows one to control singularities by induction on dimension \cite[Chapter 4]{Kol13}. Given the deep connections between singularities of the MMP and $F$-singularities, establishing inversion of adjunction for classes of $F$-singularities is an active program. For $F$-regularity and $F$-purity, various inversion of adjunction statements are known \cite{SchFad,TakWat,Inv,Das,DasSch}. Our main theorem is a quantitative inversion of adjunction statement for the $F$-signature, an important numerical invariant of singularities in positive characteristic which characterizes $F$-regularity.

Let $R$ be a ring of characteristic $p$ and $\Delta$ an effective $\Q$-divisoron $\Spec{R}$. The $F$-signature of $(R,\Delta)$, denoted $s(R,\Delta)$, incorporates subtle information about the singularities of $(R,\Delta)$ by measuring the asymptotic growth of the number of Frobenius splittings. By varying the coefficient a single component of $\Delta$, we consider the $F$-signature as a function $t \mapsto s(R, \Delta + tD)$. When $D$ is $\Q$-Cartier, the function $s(R,\Delta + tD)$ is continuous and convex \cite{BST2}. Our main theorem shows that the left derivative of $s(R,\Delta+tD)$ at $t=1$ (i.e. the rate at which the pair $(R,\Delta+tD)$ loses $F$-regularity) is determined by the $F$-signature of the pair $(\O_D,\Diff_D(\Delta))$. The $\Q$-divisor $\Diff_D(\Delta)$, called the \emph{different}, acts as a correction factor in adjunction statements (see Section \ref{subsec:Different}). For a function $f \colon \mathbb{R} \to \mathbb{R}$, we denote the left-derivative of $f$ at $t=c$ by $\frac{\partial_-}{\partial t}\big|_{t=c}f(t)$.

\begin{theorem}\label{thm:MainTheoremIntro}
  Let $(R,\Delta + D)$ be a log $\Q$-Gorenstein pair where $R$ is a Noetherian, $F$-finite, normal local ring of dimension $d$, $\Delta$ is an effective $\Q$-divisor, and $D$ is an irreducible, reduced $\Q$-Cartier divisor which is not a component of $\Delta$. Then
  \[
     \frac{\partial_-}{\partial t}\bigg|_{t=1} s(R,\Delta +tD) = -s(\O_D,\Diff_D(\Delta))
  \]
  where $\Diff_D(\Delta)$ is the different of $\Delta$ on $D$.
\end{theorem}

In the case where $R$ is regular and $\Delta = 0$, we have $\Diff_D(\Delta) = 0$, and Theorem \ref{thm:MainTheoremIntro} reduces to \cite[Theorem 4.6]{BST2}. One should interpret Theorem \ref{thm:MainTheoremIntro} as a quantitative version of inversion of adjunction for $F$-regularity. Indeed, as an immediate consequence, we obtain the following inequality.

\begin{corollary}\label{cor:InversionOfAdjunctionForFSignatureIntro}
  With notation as in Theorem \ref{thm:MainTheoremIntro}, we have
  \[
    s(R,\Delta) \geq s(\O_D,\Diff_D(\Delta))
  \]
  with equality if and only if the function $s(R, \Delta + tD)$ is linear on the interval $[0,1]$ with slope $-s(R,\Delta)$.
\end{corollary}

Since positivity of the $F$-signature characterizes $F$-regularity, Corollary \ref{cor:InversionOfAdjunctionForFSignatureIntro} recovers and refines the inversion of adjunction for $F$-regularity, i.e. the fact that the $F$-regularity of $(D,\Diff_D(\Delta))$ implies the $F$-regularity of $(R,\Delta)$ \cite{Das}. However, Corollary \ref{cor:InversionOfAdjunctionForFSignatureIntro} is stronger. Indeed, it states that, with respect to $F$-signature, the singularities of $(R,\Delta)$ are no worse than those of $(\O_D,\Diff_D(\Delta))$. In the case where $D$ is Cartier, the inequality of Corollary \ref{cor:InversionOfAdjunctionForFSignatureIntro} recovers and quantifies a classical result of Singh \cite{Singh99a}. Singh's theorem states that if $R$ is $\Q$-Gorenstein and $R/f$ is $F$-regular for some $f \in R$, then $R$ is $F$-regular. Our Corollary \ref{cor:InversionOfAdjunctionFSignatureIntro} yields the attractive inequality $s(R) \geq s(R/f)$.

The $F$-signature is notoriously difficult to compute in examples, but lower bounds often suffice for applications. Using Corollary \ref{cor:InversionOfAdjunctionFSignature}, one can bound the $F$-signature from below by computing (or bounding) the $F$-signature of certain subschemes of $\Spec{R}$ (e.g. complete intersection subschemes).

Theorem \ref{thm:MainTheoremIntro} is anticipated in the study of normalized volume. The normalized volume, defined by Li in \cite{Li18}, is a numerical invariant of klt singularities in characteristic 0. This invariant provides a local theory of stability for klt singularities parallel to the global notion of K-stability for Fano varieties. The normalized volume plays an important role in the current work on moduli of Fano varieties \cite{BlumXu}. There is mounting evidence that the normalized volume is related to $F$-signature via reduction modulo $p$. Although no formal conjectures have been stated, these invariants have been shown to exhibit similar behavior \cite{LiLiuXu,MPST}. Our main result provides further evidence of this connection. In pursuit of a theory of adjunction for normalized volume, Li, Liu, and Xu proved the following.

\begin{theorem}[{\cite[Proposition 6.8]{LiLiuXu}}]\label{thm:LiLiuXuTheoremIntro}
  Let $x \in (X,\Delta)$ be an $n$-dimensional klt singularity. Let $D$ be a normal $\Q$-Cartier divisor containing $x$ such that $(X,D+\Delta)$ is plt. Denote by $\Diff_D(\Delta)$ the different of $\Delta$ on $D$. Then
  \[
  	\frac{\partial_-}{\partial t} \bigg|_{t=1} \frac{\nvol(x,X,\Delta + tD)}{n^n} = \frac{-\nvol(x,D,\Diff_D(\Delta))}{(n-1)^{n-1}}
  \]
\end{theorem}

Up to normalization, this equality is precisely the formula in our main theorem. Thus $F$-signature and normalized volume behave identically with respect to adjunction along codimension one centers, providing further evidence that the normalized volume and $F$-signature are analogous invariants. Apart from the inherent interest in such connections, the interplay between characteristic 0 and characteristic $p$ techniques has been mutually beneficial to commutative algebra and birational geometry. It is reasonable to expect that further understanding of the relationship between the normalized volume and the $F$-signature will improve our understanding of both invariants. We hope that our result will promote further study into the connection between the normalized volume and the $F$-signature.

The paper is structured as follows. After reviewing some background material in Section \ref{sec:Background}, we apply Schwede's theory of $F$-adjunction \cite{SchFad} to the $F$-signature function in Section \ref{sec:F-Adjunction}. Section \ref{sec:VeroneseCovers} reviews the cyclic cover construction and transformation rules for $F$-signature under finite morphisms. This construction allows us to reduce to the case where $D$ is Cartier. The proof of the main theorem appears in Section \ref{sec:InversionOfAdjunction}. We conclude with Section \ref{sec:Corollaries} in which we discuss some corollaries and directions for further inquiry, particularly in the direction of understanding the relationship between $F$-signature and normalized volume.

\subsection*{Acknowledgements}
The author would like to thank his advisor, Kevin Tucker, for his support and encouragement as well as his suggestion that a version of Theorem \ref{thm:LiLiuXuTheoremIntro} might hold for $F$-signature. The author would also like to thank Javier Carvajal-Rojas for helpful comments on a previous draft.

\section{Cartier Subalgebras, $F$-Signature, and Hilbert-Kunz Multiplicity}\label{sec:Background}

\subsection{Conventions}\label{subsec:Conventions}

All rings are assumed to be a commutative with 1, Noetherian, and of characteristic $p > 0$ unless otherwise specified. If $R$ is such a ring, we have the \emph{Frobenius maps} $F^e : R \to R$ given by $r \mapsto r^{p^e}$ for each $e \geq 1$. If $I \subseteq R$, we denote by $I^{[p^e]}=(a^{p^e} : a \in I)$, the $p^e$th \emph{Frobenius power} of $I$. If $M$ is an $R$-module, we denote by $F^e_*M$ the restriction of scalars of $M$ along $F^e$. That is, $R$ acts on $F^e_*M$ via $r\cdot F_*^em = F_*^e(r^{p^e}m)$. A ring $R$ is \emph{$F$-finite} if $F_*R$ is a finite $R$-module. 
\begin{center}
\emph{In what follows, we assume that all rings of characteristic $p > 0$ are $F$-finite.}
\end{center}
Note that any ring essentially of finite type over an $F$-finite field is $F$-finite.

The \emph{index} of a $\Q$-Cartier $\Q$-divisor $\Delta$ is the smallest positive integer $m$ such that $m\Delta$ is Cartier. A pair $(R, \Delta)$ is said to be $\Q$-Gorenstein if $K_X + \Delta$ is $\Q$-Cartier on $X = \Spec{R}$, and the index of the pair $(X,\Delta)$ is the index of $K_X+\Delta$.  For divisors $D = \sum a_i D_i$ and $D' = \sum a'_i D_i$, we define $D \wedge D' = \sum \min\{a_i,a'_i\}D_i$.

\subsection{Cartier Subalgebras}\label{subsec:CartierSubalgebras}

The primary reference for this subsection is \cite{BST1}. Other useful sources include \cite{Bli,BliSch,SchTuc}. 

Let $M$ and $N$ be $R$-modules. A \emph{$p^{-e}$-linear map} from $M$ to $N$ is an $R$-linear map $\phi : F^e_*M \to N$. The abelian group $\Hom_R(F^e_*M,N)$ has an $R$-module structure via post-multiplication and an $F^e_*R$-module structure via pre-multiplication. The direct sum 
\[
  \mscr{C}^R = \bigoplus_{e \geq 0} \mscr{C}^R_e = \bigoplus_{e \geq 0}  \Hom_R(F^e_*R,R)
\] 
has a natural (non-commutative) graded $R$-algebra structure where multiplication of homogeneous elements in $\mscr{C}^R$ is given by composition. That is, if $\phi_e \in \Hom_R(F^e_*R,R)$ and $\phi_f \in \Hom_R(F^f_*R,R)$, then $\phi_e\cdot \phi_f$ is defined to be the composition
\begin{center}
  \begin{tikzcd}
    F^{e+f}_*R  \arrow[rr,"F^f_*\phi_e"] & &  F^f_*R \arrow[rr,"\phi_f"] & & R.
  \end{tikzcd}
\end{center}

\begin{definition}[{\cite[Definition 2.4]{BST1}}]\label{def:CartierAlgebra}
  The ring $\mscr{C}^R$ constructed above is the \emph{(total) Cartier algebra} of $R$. A \emph{Cartier subalgebra} on $R$ is a graded $\bb{F}_p$-subalgebra $\mscr{D} = \bigoplus_{e \geq 0} \mscr{D}_e$ of $\mscr{C}^R$ such that $\mscr{D}_0 = \Hom_R(R,R) = R$ and $\mscr{D}_e \not= 0$ for some $e >0$. For a Cartier subalgebra $\mscr{D}$, the set $\Gamma_{\mscr{D}} := \{e \in \Z_{\geq 0} : \mscr{D}_e \not= 0\}$ is a semigroup. 
\end{definition}

Our main interest is in Cartier subalgebras arising from triples $(R,\Delta,\mf{a}^t)$ where $R$ is a normal ring, $\Delta$ is an effective $\Q$-divisor on $\Spec{R}$, $\mf{a} \subseteq R$ is an ideal not contained in a minimal prime of $R$, and $t$ is a positive real number. Given such a triple $(R,\Delta,\mf{a}^t)$, we define 
\[
  \mscr{C}^{\Delta,\mf{a}^t} = \bigoplus_{e \geq 0} \mscr{C}^{\Delta,\mf{a}^t}_e = \bigoplus_{e \geq 0} \Hom_R(F_*R(\lceil (p^e-1)\Delta \rceil), R) \cdot F_*^e\mf{a}^{\lceil t(p^e-1) \rceil}.
\]
For an ideal $J \subseteq R$, we denote by $\Hom_R(F_*^eM,N)\cdot F_*^e J$ the submodule consisting of elements of the form $F_*^ea \cdot \phi$ where $a \in J$ and $\phi \in \Hom_R(F^e_*M,N)$. 

The classical definitions of $F$-regularity and $F$-purity extend to general Cartier subalgebras.

\begin{definition}[{\cite{HocRob,HocHun,SchTId,BST1}}]\label{def:F-regular_F-pure}
  Let $\mscr{D}$ be a Cartier subalgebra on a local ring $R$. We say $(R,\mscr{D})$ is \emph{(sharply) $F$-pure} if there is a surjective homomorphism in $\mscr{D}_e$ for some $e>0$. We say $(R,\mscr{D})$ is \emph{(strongly) $F$-regular} if for all $c \in R$ not contained in a minimal prime, there exists $\phi \in \mscr{D}_e$ for some $e$ such that $\phi(F_*^ec) = 1$. A triple $(R,\Delta, \mf{a}^t)$ is called \emph{$F$-pure} (resp. \emph{$F$-regular}) if the Cartier subalgebra $\mscr{C}^{\Delta, \mf{a}^t}$ is $F$-pure (resp. $F$-regular).
\end{definition}

\begin{proposition}[{\cite{AbeEne,SchCen,BST1}}]\label{prop:SplittingPrime}
  Let $\mscr{D}$ be a Cartier subalgebra on a local ring $(R,\mf{m})$. For every $e \geq 1$, let $I_e^{\mscr{D}} =\{r \in R: \phi(F^e_*r) \in \mf{m} \text{ for all } \phi \in \mscr{D}_e\}$. The ideal
  \[
    P_{\mscr{D}} := \bigcap_{e \in \Gamma_{\mscr{D}}} I_e^{\mscr{D}}.
  \] 
  is called the \emph{splitting prime} of $(R,\mscr{D})$. The ideal $P_\mscr{D}$ is proper if and only if $(R,\mscr{D})$ is $F$-pure, and in this case, $P_{\mscr{D}}$ is prime. Furthermore, $P_{\mscr{D}} = 0$ if and only if $(R,\mscr{D})$ is $F$-regular.
\end{proposition}

\begin{remark}\label{rmk:DegeneracyIdealNotation}
	The ideal $I_e^\mathscr{D}$ is sometimes called the $e$-th \emph{degeneracy ideal} of $\mathscr{D}$. When $\mathscr{D}$ is the Cartier subalgebra associated to a triple $(R,\Delta,\mf{a}^t)$, we write $I_e^{\Delta,\mf{a}^t}$ rather than $I_e^{\mathscr{D}}$.
\end{remark}

\begin{definition}[{\cite{BST1,SchCen}}]\label{def:CompatibleIdealAndInducedSubalgebra}
  Let $\mscr{D}$ be a Cartier subalgebra on $R$. An ideal $J \subseteq R$ is \emph{$\mscr{D}$-compatible} if for every $e \in \Gamma_\mscr{D}$ and every $\phi \in \mscr{D}_e$, we have $\phi(F^e_*J) \subseteq J$. In this case, there is an induced map $\phi_J : F^e_*(R/J) \to R/J$ which makes the following diagram commute
  \begin{center}
    \begin{tikzcd}
      F^e_*R \arrow[r, "\phi"] \arrow[d] & R \arrow[d] \\
      F^e_*(R/J) \arrow[r, "\phi_J"] & R/J
    \end{tikzcd}
  \end{center}
  where the vertical arrows are the natural quotient maps. The \emph{induced Cartier subalgebra} $\mscr{D}_J$ on $R/ J$ is the Cartier subalgebra whose $e$th graded piece is $(\mscr{D}_J)_e := \{\phi_J : \phi \in \mscr{D}_e\}$.
\end{definition}

\begin{remark}\label{rmk:SplittingPrimeLargestCompatibleIdeal}
  For a pair $(R,\mscr{D})$, the splitting prime $P_{\mscr{D}}$ is the largest $\mscr{D}$-compatible ideal. The subscheme defined by $P_{\mscr{D}}$ is called the \emph{minimal center of $F$-purity} \cite{SchCen}.
\end{remark}

\begin{example}[{cf. \cite[Section 7]{SchFad}}]\label{ex:mainSetupExample}
  Suppose $(R,\Delta)$ is a pair where $R$ is a normal local domain and $\Delta$ is an effective $\Q$-divisor. Suppose $D$ is a normal, reduced, irreducible $\Q$-Cartier divisor on $\Spec{R}$ which is not a component of $\Supp\Delta$. Let $\mf{p}$ be the ideal defining $D$. Suppose $K_X+\Delta + D$ is $\Q$-Cartier with index prime to $p$. The
   pair $(X,\Delta+D)$ is $F$-pure at the generic point of $D$ since $X$ is normal (and hence regular in codimension 1). Also $\mf{p}$ is $\mscr{C}^{\Delta + D}$-compatible, so $\mf{p} \subseteq P_{\mscr{C}^{\Delta + D}}$.
\end{example}

Before concluding this subsection, we record an inclusion of ideals which will be used in the proof of the main theorem.

\begin{lemma}[{cf. \cite[Lemma 4.4]{Tuck}}]\label{lemma:ColonFrobeniusPowerContainment}
  Suppose $R$ is a reduced local ring and $\Delta$ is an effective $\Q$-divisor on $\Spec{R}$. For any element $f \in R$ not in a minimal prime of $R$ and positive integers $r \geq e$, we have 
  \[(I^\Delta_e : f)^{[p^{r-e}]} \subseteq (I^\Delta_r : f^{p^{r-e}}).\]
\end{lemma}
\begin{proof}
  If $a \in (I^\Delta_e:f)$ and $\phi \in \Hom_R(F_*^rR(\lceil (p^r-1)\Delta \rceil),R)$, then
  \[
    \phi(F_*^r(f^{p^{r-e}}a^{p^{r-e}})) = \phi|_{F_*^eR}(F_*^e(fa)) \in \mf{m}.
  \]
  as required. 
\end{proof}

\subsection{$F$-Signature and $F$-Splitting Ratio}\label{subsec:F-Signature}

Let $R$ be a local ring of characteristic $p>0$. The $F$-signature is a numerical invariant first implicitly considered in \cite{SmiVdB} and formally defined in \cite{HunLeu} (see \cite{Hun} for a survey). The $F$-signature measures the severity of the singularity of $R$ by recording the asymptotic behavior of the number of splittings of $F^e_*R$ as an $R$-module. To measure singularities of a pair $(R,\mscr{D})$ where $\mscr{D}$ is a Cartier subalgebra, we count those splittings that lie in $\mscr{D}$. More precisely, a free direct summand $R^{\oplus n} \subseteq F^e_*R$ is called a \emph{$\mscr{D}$-summand} if each of the associated projection maps $\phi \in \Hom_R(F_*^eR,R)$ lies in $\mscr{D}_e$.

\begin{definition}[{\cite{BST1,BST2,Tuck}}]\label{def:F-signature}
  The \emph{$e$-th $F$-splitting number} of a pair $(R,\mscr{D})$, denoted $a_e^{\mscr{D}}$, is the maximal rank of a $\mscr{D}$-summand of $F^e_*R$ as an $R$-module. If $d = \dim{R}$ and $\alpha(R) = \log_p([k:k^p])$, the limit
  \[
    s(R,\mscr{D}) = \lim_{e \to \infty} \frac{a_e^{\mscr{D}}}{p^{e(d+\alpha(R))}}
  \]
  exists and is called the \emph{$F$-signature} of $(R,\mscr{D})$. When $\mscr{D} = \mscr{C}^{\Delta,\mf{a}^t}$ is the Cartier subalgebra associated to a triple $(R,\Delta, \mf{a}^t)$, we write $s(R,\Delta, \mf{a}^t)$.
\end{definition}

The next proposition is useful for computations. One would not lose much in what follows by taking it as the definition of $F$-signature.

\begin{proposition}[{\cite[Proposition 2.2]{BST2}}]\label{prop:F-signatureInTermsOfLength}
  Let $(R,\Delta,\mf{a}^t)$ be a triple. Then 
  \[
    s(R,\Delta,\mf{a}^t) = \lim_{e \to \infty} \frac{1}{p^{ed}}\ell_R \left( \frac{R}{(I_e^\Delta : \mf{a}^{\lceil t(p^e-1) \rceil})} \right).
  \]
  We also have
  \[
    s(R,\Delta,\mf{a}^t) = \lim_{e \to \infty} \frac{1}{p^{ed}}\ell_R \left( \frac{R}{(I_e^\Delta : \mf{a}^{\lceil tp^e \rceil})} \right).
  \]
\end{proposition}

As stated previously, the $F$-signature measures singularities, with smaller $F$-signature corresponding to worse singularities.

\begin{theorem}\label{thm:PositivityOfF-signature}
  Let $R$ be a local ring and $\mscr{D}$ a Cartier algebra on $R$. 
  \begin{enumerate}
    \item \emph{\cite[Corollary 16]{HunLeu}} The ring $R$ is regular if and only if $s(R) = 1$.
    \item \emph{\cite{AbeLeu,BST1}} The pair $(R,\mscr{D})$ is $F$-regular if and only if $s(R,\mscr{D}) > 0$.
  \end{enumerate}
\end{theorem}

In the case $(R,\mscr{D})$ is $F$-pure but not $F$-regular (so $s(R,\mscr{D}) = 0)$, there is still a numerical invariant to consider.

\begin{theorem}[{\cite[Theorem 4.2]{BST1}}]\label{thm:F-splittingRatio}
  Let $\mscr{D}$ be a Cartier subalgebra on a local ring $R$. Suppose $(R,\mscr{D})$ is $F$-pure with $F$-splitting prime $P_\mscr{D}$. Then the limit
  \[
    r_F(R,\mscr{D}) = \lim_{e \in \Gamma_{\mscr{D}} \to \infty} \frac{1}{p^{e\dim{R/P_\mscr{D}}}} \ell_R\left( R/I_e^\mscr{D} \right)
  \]
  exists and is called the \emph{$F$-splitting ratio}. The integer $\dim{R / P_{\mscr{D}}}$ is called the \emph{splitting dimension} of $(R,\mscr{D}$). Furthermore, if $\mscr{D}_{P_\mscr{D}}$ denotes the induced Cartier subalgebra on $R/P_{\mscr{D}}$, then
  \[
    r_F(R,\mscr{D}) = s(R/P_{\mscr{D}}, \mscr{D}_{P_\mscr{D}}) 
  \]
  so that, in particular, $0 < r_F(R,\mscr{D}) \leq 1$.
\end{theorem}

Our main theorem concerns the behavior of $s(R,\Delta + tD)$ as a function of the real parameter $t$. Blickle, Schwede, and Tucker established some formal properties of the function $s(R,\Delta, \mf{a}^t)$ in \cite{BST2}.

\begin{proposition}[{\cite[Section 3]{BST2}}]\label{prop:ContinuityAndConvexityOfF-signature}
  Suppose $R$ is a normal local domain, $\Delta$ is an effective $\R$-divisor on $\Spec{R}$, and $D$ is a $\Q$-Cartier divisor. Then the function $s_R^{\Delta,D}(t) := s(R,\Delta + tD)$ is continuous and convex in $t$. In particular, $\frac{\partial_-}{\partial t}\big|_{t=1} s_R^{\Delta,D}(t)$ exists.
\end{proposition}
\begin{proof}
  Suppose $nD$ is Cartier with defining equation $f \in R$. Then we have
  \[
    s(R, \Delta + tD) = s\left(R, \Delta + \frac{t}{n}\dv{f}\right) = s(R, \Delta, f^{t/n}).
  \]
  From \cite[Theorem 3.9]{BST2}, we have that $s(R,\Delta,\mf{a}^t)$ is continuous. Furthermore, it is convex if $\mf{a}$ is principal. So $s(R,\Delta+tD)$ is continuous and convex, as it is equal to $h(t/n)$ for a continuous and convex function $h$. 
\end{proof}

\subsection{Hilbert-Kunz Multiplicity}\label{subsec:Hilbert_Kunz}

\begin{definition}\label{def:HilbertKunzMultiplicity}
  Let $I \subseteq R$ be an ideal of finite colength. We define the \emph{Hilbert-Kunz multiplicity} of $I$ along $R$ by
  \[
    \HK(I;R) := \lim_{e \to \infty} \frac{1}{p^{ed}} \ell_R \left( R/I^{[p^e]} \right).
  \] 
\end{definition}
Monsky proved that this limit exists in \cite{Mon}. The first proof of the existence of $F$-signature in general used the existence of Hilbert-Kunz multiplicity \cite{Tuck}, establishing the following precise relationship between the two invariants.

\begin{theorem}[{\cite{Tuck,BST1}}]\label{thm:F-signatureAndHK}
  Suppose $\mscr{D}$ is a Cartier algebra on a local ring $R$ of dimension $d$. Then
  \[
    s(R,\mscr{D}) = \lim_{e \in \Gamma_{\mscr{D}} \to \infty} \frac{1}{p^{ed}} \HK(I_e^\mscr{D}, R).
  \]
\end{theorem}

\section{$F$-Adjunction}\label{sec:F-Adjunction}

\subsection{The Different}\label{subsec:Different}

Let $(R,\Delta)$ be a pair where $R$ is a normal, local domain with an effective $\Q$-divisor $\Delta$ on $\Spec{R}$. Let $D$ be a normal, prime $\Q$-Cartier divisor on $\Spec{R}$ such that $\Delta \wedge D = 0$. Furthermore, assume that the pair $(R,\Delta + D)$ is log $\Q$-Gorenstein with index prime to $p$. When $X$ and $D$ are regular, the classical adjunction formula states that $(K_X + D)|_D = K_D$. In general, we have $(K_X + \Delta + D)|_D = K_D + \Diff_D(\Delta)$, where the correction factor $\Diff_D(\Delta)$ is called the \emph{different} of $\Delta$ on $D$. In what follows, we construct this divisor via $p^{-e}$-linear maps.

As laid out explicitly in \cite{SchFad}, for any normal $X = \Spec{R}$, there is a bijection
\begin{equation}\label{eq:DivisorMapCorrespondence}
  \left\{ \substack{\text{Effective }\Q\text{-divisors }\Delta \text{ such} \\ \text{that }(p^e-1)(K_X+\Delta) \text{ is Cartier}} \right\} \leftrightarrow \left\{ \text{Non-zero elements of }\Hom_{R}(F_*^eR, R) \right\} \big/ \sim
\end{equation}
where two maps are $\sim$-equivalent if they differ by pre-multiplication by a unit of $R$. Indeed, Grothendieck duality for a finite map gives an isomorphism $\Hom_{R}(F_*^eR,R) \cong R((1-p^e)K_X)$ \cite{ResAndDua}. Given a map $\phi : F_*^eR \to R$, we associate an effective divisor $D_\phi \sim (1-p^e)K_X$. We normalize $D_\phi$ by defining $\Delta_\phi := \left( \frac{1}{p^e-1} \right)D_\phi$. See \cite[Section 4]{BliSch} for the details of this correspondence and its generalizations.

\begin{definition}[{\cite{SchFad}}]\label{def:FDifferent} 
  With notation as above, $(p^e-1)(K_X+\Delta+D)$ is Cartier for some $e$, so $\Delta+D$ corresponds to an $R$-linear map $\phi_{\Delta+D} : F^e_*R \to R$. Since $D$ is a $\mscr{C}^{\Delta+D}$-compatible subscheme, there is an induced map $\overline{\phi_{\Delta+D}} : F^e_*\O_D \to \O_D$ making the following diagram commute
  \begin{center}
    \begin{tikzcd}
      F^e_*R \arrow[r, "\phi_{\Delta+D}"] \arrow[d] & R \arrow[d] \\
      F^e_*\O_D \arrow[r, "\overline{\phi_{\Delta+D}}"] & \O_D.
    \end{tikzcd}
  \end{center}
  Since $D$ is normal, the correspondence \eqref{eq:DivisorMapCorrespondence} holds for $D$, and the divisor associated to $\overline{\phi_{\Delta+D}}$, denoted $\Diff_D(\Delta)$, is called the \emph{different of $\Delta$ on $D$}.
\end{definition}

The proposition below lists the relevant properties of the different. Both follow quickly from construction and are proven carefully in \cite{SchFad}.

\begin{proposition}\label{prop:PropertiesOfTheDifferent}
  Let $(X,\Delta + D)$ be a log $\Q$-Gorenstein pair with index prime to $p$ where $D$ is a normal, irreducible $\Q$-Cartier divisor. Suppose $(p^{e_0} - 1)(K_X + \Delta + D)$ is Cartier.
  \begin{enumerate}
    \item $(p^{e_0}-1)(K_D + \Diff_D(\Delta))$ is Cartier.
    \item The projection map
    \[
      \Hom_R(F_*^eR((p^e-1)(\Delta + D)), R) \to \Hom_{\O_D}(F_*^e\O_D((p^e-1)\Diff_D(\Delta)), \O_D)
    \]
    is surjective for all $e$ such that $e_0 | e$.
  \end{enumerate}
\end{proposition}

\begin{remark}
  This construction of the different was introduced in \cite{SchFad}, where it was called the $F$-different, in order to formulate adjunction statements in the context of $F$-singularities. In birational geometry, there is a divisor called the different which appears in the theory of adjunction \cite[Chapter 4]{Kol13}. In our case (i.e. where $D$ is a divisor), the different and the $F$-different are equal \cite[Theorem 5.3]{Das}, so for the rest of the paper, we drop the ``$F$-" prefix. It should be noted however that for higher codimension centers of $F$-purity, the $F$-different and the different do not coincide in general \cite{DasSch}.
\end{remark}

The proof of our main theorem only uses the construction of Definition \ref{def:FDifferent}, but for completeness, we briefly review the geometric construction of the different closely following \cite[Chapter 4]{Kol13}. For the most efficient path to the proof of the main theorem, the reader can skip directly to Section \ref{subsec:F-Adjunction}.

Let $X$ be a normal variety, $D$ a reduced, effective Weil divisor and $\Delta$ an effective $\Q$-divisor that shares no components with $D$. Suppose $m(K_X + D + \Delta)$ is Cartier for some $m\geq 1$, and let $\nu : D' \to D$ be the normalization map\footnote{The different exists in greater generally. See \cite[Chapter 4]{Kol13} for a discussion under a set of minimal assumptions.}. Let $Z \subseteq D$ be the union of the non-regular locus of $D$ and $D \cap \Supp(\Delta)$. Note that $Z$ has codimension at least 1 in $D$ by assumption. Since $D \setminus Z$ is regular, the Poincar\'{e} residue map provides an isomorphism
\[
  \mcal{R}_D : \omega_X(D)|_{D\setminus Z} \map{\sim} \omega_{D}|_{D \setminus Z}.
\]
Moreover, $\nu|_{D' \setminus \nu^{-1}(Z)} : D' \setminus \nu^{-1}(Z) \to D \setminus Z$ is an isomorphism and $\Supp(\Delta) \cap D \subseteq Z$. So taking reflexive powers\footnote{For a sheaf $\mathcal{F}$, the sheaf $\mathcal{F}^{[m]}$ is the reflexification (or S2-ification) of the $\mathcal{F}^{\otimes m}$. That is $\mathcal{F}^{[m]} := (\mathcal{F}^{\otimes m})^{**}$. This is the only time we use the reflexive powers notation, so it should not be confused with the notation for the Frobenius power of an ideal.} of $\mcal{R}_D$ yields an isomorphism
\[
  \mcal{R}_D^m : \nu^*\left( \omega_X^{[m]}(mD + m\Delta) \right)\big|_{D' \setminus \nu^{-1}(Z)} \map{\sim} \omega^{[m]}_{D'}\big|_{D' \setminus \nu^{-1}(Z)}.
\]
Let $D'_{reg} \subseteq D'$ be the regular locus of $D'$. So $\nu^*\omega_{X}^{[m]}(mD+m\Delta)$ and $\omega_{D'}^{[m]}$ are invertible sheaves on $D'_{reg}$ which implies
\[
  \sHom_{\O_{D'}}\left( \nu^*\left( \omega_X^{[m]}(mD+m\Delta) \right), \omega_{D'}^{[m]} \right)
\]
is invertible on $D'_{reg}$. Furthermore, $\mcal{R}_D^m$ is a rational section of this invertible sheaf. As such, there is a unique divisor $\Delta_{D'_{reg}}$ on $D'_{reg}$ such that $\mcal{R}_D^m$ extends to an isomorphism
\begin{equation}\label{eq:DiffConstruction}
  \mcal{R}_{D'_{red}}^m : \nu^* \left( \omega_X^{[m]}(mD+m\Delta) \right) \map{\sim} \omega_{D'}^{[m]}(\Delta_{D'_{reg}})
\end{equation}
on all of $D'_{reg}$. As $D'$ is normal, $D' \setminus D'_{reg}$ has codimension at least 2 in $D'$, and $\Delta_{D'_{reg}}$ extends to a uniquely to a divisor $\Delta_{D'}$ on $D'$. The different of $\Delta$ on $D$ is defined to be $\Diff_{D'}(\Delta) = \frac{1}{m}\Delta_{D'}$. Note that $(K_X + D + \Delta)|_{D'} \sim_{\Q} K_{D'} + \Diff_{D'}(\Delta)$.

When it is defined, $\Diff_D(0)$ can be interpreted as a measure of how far $D$ is from being Cartier.

\begin{example}\label{ex:DifferentCartierCase}
  Let $R$ be a normal local domain with an effective Cartier integral divisor $D = \dv(f)$. The inverse linear map associated to the zero divisor is any generator $\Phi^e$ of $\Hom_R(F_*^eR,R)$. Now $D$ is compatible with the map $\Phi^e(F_*^e(f^{p^e-1}\cdot -))$, and the induced map $\phi^e : F_*^e R/f \to R/f$ generates $\Hom_{R/f}(F_*^e R/f, R/f)$ \cite[Example 4.3.2]{BST1}. So the divisor associated to $\phi^e$ is 0, but this is the different by definition.
\end{example}

Below, we see an example where $\Diff_D(0)$ is non-zero even when the divisor $D$ is regular. In particular, $\Diff_D(0)$ depends on the embedding of $D$ in $\Spec{R}$ rather than the geometry of $D$ itself.

\begin{example}\label{ex:AnSingularities}
  ($A_n$ singularities) Let $R = S/f = \bb{F}_p\llbracket x,y,z\rrbracket / (xy-z^{n+1})$ with $n\geq 1$ and $p \geq 3$, and let $X= \Spec{R}$. Let $D = V(x,z)$, so $D$ is $\Q$-Cartier with index $n+1$. Indeed, $(n+1)D = \dv{x}$. We compute the different $\Diff_D(0)$ in two ways. First, we use the geometric construction. This computation also appears in \cite[Example 4.3]{Kol13}. The canonical module $\omega_X$ is generated by $\frac{1}{y} \,dy \wedge dz$. So $\omega_X^{[n+1]}((n+1)D)$ is locally free with generator
  \[
    \frac{(dy \wedge dz)^{\otimes (n+1)}}{xy^{n+1}} = \frac{(dy \wedge dz)^{\otimes (n+1)}}{z^{n+1}y^{n}} = \frac{(dz)^{\otimes n+1}}{z^{n+1}} \wedge \frac{dy^{\otimes {n+1}}}{y^{n}}.
  \]
  The residue of the generator is $dy^{\otimes n+1} / y^n$. So the different is $\Diff_D(0) = (1 - 1/(n+1)) [0]$ where $[0]$ is the class of the point $0 \in D \cong \A^1$ with coordinate $y$. Note that $\Diff_D(0)$ is non-zero despite the fact that $D$ is regular.
  
  Now, let us compute the different via Frobenius. With this approach, we require that the index of $D$ be prime to $p$, so there is some $e$ such that $n+1 | p^e-1$. Let $\Phi^e: F_*^e S \to S$ be the map defined by 
  \[
    \Phi^e(x^{i}y^jz^k) = \begin{cases}
      1 & i=j=k=p^e-1 \\
      0 & \text{otherwise}
    \end{cases}
  \]
  for $0 \leq i,j,k \leq p^e-1$. This map generates $\Hom_S(F_*^eS,S)$ as an $F_*^eS$-module. Note that $\Phi^e(f^{p^e-1}\cdot -)$ induces a map $\phi^e : F_*^e R \to R$ that generates $\Hom_R(F_*^eR, R)$ (see \cite[Example 4.3.2]{BST1}). The map associated to the divisor $D$ is $\phi^e(x^{(p^e-1)/(n+1)}\cdot -)$. Since $D$ is $\phi^e(x^{(p^e-1)/(n+1)} \cdot -)$-compatible, there is an induced map $\phi^e_{\Diff_D(0)} : F_*^e\O_D \to \O_D$. We summarize the setup in the following diagram.
  \begin{center}
    \begin{tikzcd}
      F_*^eS \arrow[rr] \arrow[d, "\Phi^e(F_*^e(f^{p^e-1}x^{\frac{p^e-1}{n+1}}\cdot -))"'] & & F_*^eR \arrow[rr] \arrow[d, "\phi^e(F_*^e(x^{\frac{p^e-1}{n+1}}\cdot -))"] & & F_*^e\bb{F}_p\llbracket y \rrbracket \arrow[d, "\phi_{\Diff_D(0)}"] \\
      S \arrow[rr] & & R \arrow[rr] & & \bb{F}_p\llbracket y \rrbracket
    \end{tikzcd}
  \end{center}
  Since $\O_D = \bb{F}_p\llbracket y\rrbracket$, the map $\phi_{\Diff_D(0)}$ is determined by its values on $F_*^ey^i$ for $0 \leq i \leq p^e-1$. Thus, $\phi_{\Diff_D(0)}(F_*^ey^i) = \Phi^e(F_*^e( f^{p^e-1}x^{\frac{p^e-1}{n+1}}y^i)) \mod{(x,z)}$. So we compute
  \begin{align*}
    \Phi^e\left(F_*^e( f^{p^e-1}x^{\frac{p^e-1}{n+1}}y^i) \right) &= \Phi^e\left(  F_*^e \left( \sum_{k=0}^{p^e-1} {p^e-1 \choose k} x^{k + \frac{p^e-1}{n+1}}y^{k+i}z^{(n+1)(p^e-1-k)} \right)\right) \\
    &\equiv \Phi^e\left( F_*^e {p^e-1 \choose \frac{n(p^e-1)}{n+1}} x^{p^e-1}y^{\frac{n(p^e-1)}{n+1} + i} z^{p^e-1} \right) \mod{(x,z)} \\
    &\equiv \begin{cases}
      {p^e-1 \choose n(p^e-1)/(n+1)}  & i = \frac{p^e-1}{n+1} \\
      0  & \text{otherwise}
    \end{cases} \mod{(x,z)}
  \end{align*}
  Since this binomial coefficient is a unit in $\bb{F}_p\llbracket y \rrbracket$, we see that the divisor associated to this map is
  \[
    \Diff_D(0) = \frac{1}{p^e-1} \left( \frac{n(p^e-1)}{n+1}\right)\dv{y} = \left(1-\frac{1}{n+1}\right)[0].
  \]
\end{example}

\subsection{$F$-Adjunction for Cartier Subalgebras}\label{subsec:F-Adjunction}

In this subsection, we specify the relationship between $\mscr{C}^{\Delta + D}$ on $R$ and $\mscr{C}^{\Diff_D(\Delta)}$ on $\O_D$.

\begin{proposition}\label{prop:CartierSubalgebraComparison}
  Let $(R,\Delta + D)$ be a log $\Q$-Gorenstein pair with index prime to $p$. Assume $R$ is a normal local ring, $\Delta$ is an effective $\Q$-divisor, and $D$ is an irreducible, reduced, normal, $\Q$-Cartier divisor with $D \wedge \Delta = 0$.
  \begin{enumerate}
    \item The $e$th graded piece of $\mscr{C}^{\Diff_D(\Delta)}$ is equal to the $e$th graded piece of the Cartier subalgebra on $\O_D$ induced by $\mscr{C}^{\Delta+D}$ (as in Definition \ref{def:CompatibleIdealAndInducedSubalgebra}) for sufficiently divisible $e$.
    \item For sufficiently divisible $e$, the ideal $I_e^{\Diff_D(\Delta)}$ is the extension of $I_e^{\Delta+D} $ in $\O_D$. In particular, $\ell_R(R/I_e^{\Delta+D}) = \ell_{\O_D}(\O_D / I_e^{\Diff_D(\Delta)})$ for sufficiently divisible $e$.
  \end{enumerate} 
\end{proposition}
\begin{proof}
  Assume $(p^{e_0}-1)(K_X + D + \Delta)$ is Cartier. The $e$th graded piece of the Cartier subalgebra induced by $\mscr{C}^{\Delta+D}$ is the collection of maps $\overline{\psi}: F_*^e\O_D \to \O_D$ which fit into a diagram 
  \begin{center}
    \begin{tikzcd}
      F^e_*R \arrow[d] \arrow[r, "\psi"] & R \arrow[d] \\
      F^e_*\O_D \arrow[r, "\overline{\psi}"] & \O_D
    \end{tikzcd}
  \end{center}
  for some $\psi : F^e_*R \to R$, i.e. it is the image of the projection map
  \[
    \Hom_R(F^e_*R((p^e-1)(\Delta + D)), R) \to \Hom_{\O_D}(F_*^e\O_D((p^e-1)\Diff_D(\Delta)), \O_D).
  \]
  By Proposition \ref{prop:PropertiesOfTheDifferent}, this map is surjective for all $e$ such that $e_0 | e$. So the first statement follows.
  
  Part (2) follows from (1) given the definition of the ideals $I_e^{\Diff_D(\Delta)}$ and $I_e^{\Delta+D}$. The length statement uses the fact since $D$ is $\mscr{C}^{\Delta+D}$-compatible, so its ideal is contained in each $I_e^{\Delta+D}$. 
\end{proof}

The following consequence of Proposition \ref{prop:CartierSubalgebraComparison} says that we may compute the $F$-signature of the pair $(\O_D,\Diff_D(\Delta))$ from the pair $(R,\Delta+D)$. Our proof of Theorem \ref{thm:MainTheorem} only requires the case where $D$ is Cartier, but we include the general case since its proof is no harder.

\begin{lemma}\label{lem:LimitLemma}
  Let $(R,\Delta + D)$ be a log $\Q$-Gorenstein pair with index prime to $p$. Assume $R$ is a normal local ring of dimension $d$, $\Delta$ is an effective $\Q$-divisor, and $D$ is a normal $\Q$-Cartier divisor of index $m$ such that $D \wedge \Delta= 0$ and $(m,p)=1$. If $mD = \dv(f)$, and $m | (p^{e_0}-1)$, then
  \[
    s(\O_D, \Diff_D(\Delta)) = \lim_{\substack{e \to \infty \\ e_0|e}} \frac{1}{p^{e(d-1)}}\ell_R\left( \frac{R}{(I_e^\Delta : f^{(p^e-1)/m})}  \right).
  \]
  If $D$ is the minimal center of $F$-purity of the pair $(R,\Delta+D)$ then $s(\O_D, \Diff_D(\Delta)) = r(R,\Delta+D)$. Otherwise $s(\O_D,\Diff_D(\Delta)) = 0$.
\end{lemma}
\begin{proof}
  The fact that $s(\O_D,\Diff_D(\Delta)) = 0$ when $D$ is not the minimal center of $F$-purity is an immediate consequence of Theorem \ref{thm:F-splittingRatio}. To see that the limit above computes $s(\O_D,\Diff_D(\Delta)$, note that by definition 
  \[
    s(\O_D, \Diff_D(\Delta)) = \lim_{e \to \infty} \frac{1}{p^{e(d-1)}}\ell_{\O_D}\left( \frac{\O_D}{I_e^{\Diff_D(\Delta)}} \right).
  \] 
  On the other hand, Proposition \ref{prop:CartierSubalgebraComparison} implies that
  \[
    \ell_{\O_D} \left( \frac{\O_D}{I_e^{\Diff_D(\Delta)}} \right) = \ell_{R}\left( \frac{R}{I_e^{\Delta + D}} \right)
  \]
  for large and divisible $e$. Furthermore for $e_0 | e$, we have 
  \begin{align*}
    \mscr{C}^{\Delta+D}_e &= \Hom_R(F_*^eR(\lceil(p^e-1)(\Delta+D) \rceil), R) \\
    &= \Hom_R\left(F_*^eR\left(\lceil (p^e-1)\Delta \rceil + \frac{p^e-1}{m}\dv(f)\right), R\right) \\
    &= \Hom_R(F_*^eR(\lceil (p^e-1)\Delta \rceil), R) \cdot F_*^e (f)^{(p^e-1)/m} \\
    &= \mscr{C}^{\Delta, f^{1/m}}_e.
  \end{align*}
  Hence $I_e^{\Delta + D} = (I_e^\Delta : f^{(p^e-1)/m})$ for $e_0|e$, and the proposition holds as long as the stated limit exists. However, $\ell_R(R/(I_e^\Delta:f^{(p^e-1)/m}))$ grows like $p^{ed'}$, where $d' = \dim{R / P_{\mscr{C}^{\Delta+D}}}$ is the splitting dimension of $(R,\Delta+D)$ (see \cite{BST1}). The ideal defining $D$ is $\mscr{C}^{\Delta+D}$-compatible, so $d' \leq d-1$ which implies that the limit exists. 
\end{proof}

\section{Cyclic Covers}\label{sec:VeroneseCovers}

To prove Theorem \ref{thm:MainTheorem}, we use a cyclic cover to reduce to the case where $D$ is Cartier. In this section, we review the cyclic cover construction and apply transformation rules for $F$-signature for finite maps \cite{CRST,Car,CRThesis}.

Let $(R,\mf{m})$ be a normal, local ring of characteristic $p > 0$. Let $D$ be a prime $\Q$-Cartier divisor of index $n$ with corresponding height 1 prime $\mf{p}$. So $nD = \div(f)$ for some $f \in R$. Consider the $\Z/n\Z$-graded ring $C(D) := \bigoplus_{i=0}^{n-1} R(-iD)t^i = \bigoplus_{i=0}^{n-1} \mf{p}^{(i)}t^i$ where $t^n = 1/f$ and multiplication is given by the natural pairing $R(iD) \otimes R(jD) \to R((i+j)D)$. Let $\pi : \Spec{C(D)} \to \Spec{R}$ be the finite map of schemes associated to the inclusion $R$ into $C(D)$ as the $i=0$ summand. This map is called the \emph{cyclic cover} of $(R,D)$. See \cite{TomWat} for a readable introduction to this construction.

\begin{proposition}\label{prop:CyclicCoverProperties}
  Let $(R,\mf{m})$, $D$, and $C(D)$ be defined as above.
  \begin{enumerate}
    \item The ring $C(D)$ is local with maximal ideal $\mf{n} = \mf{m} \oplus \left(\bigoplus_{i=1}^{n-1} R(-iD) \right)$, and $\pi$ is a finite map of degree $n$.
    \item If $R$ is strongly $F$-regular, then $C(D)$ is strongly $F$-regular.
    \item The divisor $D' := \pi^*D$ is a Cartier divisor on $\Spec{C(D)}$.
    \item The restriction $\pi|_{D'} : D' \to D$ is a finite cover of degree $n$.
  \end{enumerate}  
\end{proposition}
\begin{proof}
  For (1), the maximal ideal of $C(D)$ is computed in \cite[Proposition 5.3.1]{CRThesis}. Since each $R(-iD)$ is a rank one reflexive module, we see that $\Frac{C(D)} = C(D) \otimes_R \Frac{R} \cong (\Frac{R})^{\oplus n}$, so the map is finite of degree $n$ (see \cite{TomWat} for more details).
  
  Part (2) is \cite[Corollary 6.3.1]{CRThesis}.
  
  For (3), let $nD = \div(f)$. Then $(ft^{n-1})^n = f^n / f^{n-1} = f$ in $C(D)$. So
  \[
    n\cdot \pi^*D = \pi^* \div_R(f) = \div_{C(D)}(f) = \div_{C(D)}((ft^{n-1})^n) = n\cdot \div_{C(D)}(ft^{n-1}).
  \]
  Dividing both sides by $n$ gives $\pi^*D = \div_{C(D)}(ft^{n-1})$ as required. 
  
  For (4), we compute the ring $\O_{D'}$ explicitly. Let $\mf{p}$ be the ideal of $D$ in $R$. By (3), the ideal defining $D'$ is $(ft^{n-1})$. We claim that $\O_{D'} = R/\mf{p} \oplus \left(\mf{p} / \mf{p}^{(2)}\right)t \oplus \cdots \oplus \left(\mf{p}^{(n-1)} / \mf{p}^{(n)}\right)t^{n-1}$. Clearly $(ft^{n-1}) \cap R = \mf{p}$. For the higher degree terms, any element of $(ft^{n-1}) \cap \mf{p}^{(j)}$ is of the form $(ft^{n-1})(at^{j+1})$ for $a \in \mf{p}^{j+1}$. Then $(ft^{n-1})(at^{j+1}) = aft^{n+j} = aft^j/f = at^j \in \mf{p}^{(j+1)}t^j$. For the reverse inclusion, if $at^j \in \mf{p}^{(j+1)}t^j$, then $at^j = at^{j+1}\cdot ft^{n-1} \in (ft^{n-1})$. The map $D' \to D$ is induced by the inclusion of $R / \mf{p}$ into the degree 0 piece of $\O_{D'}$, and the result follows. 
\end{proof}

The behavior of $F$-signature under finite maps is well understood \cite{CRST,Car,CRThesis}. We will use the following transformation rule to lift our calculation to the cyclic cover.

\begin{theorem}[{\cite[Theorem 3.0.1]{CRThesis}, \cite[Theorem 4.9]{Car}}]\label{thm:TransformationRuleF-Sig}
  Let $(R,\mf{m}) \subseteq (S,\mf{n})$ be a finite local extension of normal domains with corresponding morphism of schemes $f : Y \to X$. Let $\Delta$ be an effective $\Q$-divisor on $X$. Suppose there is a nonzero morphism of $S$-modules $\tau : S \to \Hom_R(f_*S, R) = \omega_{S/R}$ such that $T:= \tau(1)$ is surjective, $T(\mf{n}) \subseteq \mf{m}$ and $\Delta^* = f^*\Delta - D_T$ is effective on $Y$. Then 
  \[
    [\kappa(\mf{n}):\kappa(\mf{m})] \cdot s(S,\Delta^*) = [\Frac{S}:\Frac{R}]\cdot s(R,\Delta).
  \]
  Furthermore, if $(R,\mf{m}) \subseteq (S,\mf{n})$ is merely a local extension where $R$ is a domain and $S$ is a reflexive $R$-module and $\tau$ is an isomorphism, then
  \[
    [\kappa(\mf{n}), \kappa(\mf{m})]\cdot s(S,\Delta^*) = \dim_KS_K\cdot s(R)
  \]
  where $K \to S_K$ is the generic fiber of $R \subseteq S$.
\end{theorem}

\begin{remark}\label{rmk:WhatDoesTDo}
  This theorem is more general than the one proven in \cite{CRST} which assumes that the extension is separable. In the separable case, the field trace $\Tr_{L/K} : L \to K$ restricts to an $R$-linear map $\Tr_{S/R} : S \to R$. Theorem \ref{thm:TransformationRuleF-Sig} generalizes \cite{CRST} by replacing $\Tr_{S/R}$ with the map $T$ which satisfies similar formal properties. This extra generality allows us to extend our theorem to the case where $p$ divides the index of $D$. Indeed for cyclic covers, we have the map $T : C(D) \to R$ given by projection onto the degree 0 summand. This divisor generates $\Hom_R(C(D), R)$ as a $C(D)$-module, so $D_T = 0$. For further details, see \cite{CRThesis}.
\end{remark}

\begin{remark}\label{rmk:GeneralCartierTransformationRule}
  Other invariants of Cartier subalgebras seem to behave well with respect to finite morphisms. In \cite{CarSta}, Carvajal-Rojas and St\"{a}bler systematically treat the behavior of $F$-signature, $F$-splitting ratio, splitting primes, and the test ideal of arbitrary Cartier subalgebras under finite maps.
\end{remark}

In order to transfer our problem to a cyclic cover, we must understand how the different behaves in this construction. To accomplish this, we need the following transposition criterion \cite[Theorem 5.7]{SchTuckFiniteMaps}. In the setting of Theorem \ref{thm:TransformationRuleF-Sig}, we say that $\psi : F^e_*S \to S$ is a \emph{transpose} of $\phi: F_*^e R \to R$ along $T : S \to R$ if the following diagram commutes
\begin{center}
  \begin{tikzcd}
    F_*^eS \arrow[r, "\psi"] \arrow[d, "F_*^eT"']& S \arrow[d, "T"] \\
    F_*eR \arrow[r, "\phi"] & R.
  \end{tikzcd}
\end{center}
In the following theorem, we denote the divisor associated to a map $\phi : F_*^e R \to R$ by $\Delta_\phi$. If $R \subseteq S$ is a finite inclusion of rings, then an $R$-linear map $T :S \to R$ yields a divisor $D_T$ on $S$ via duality as in the correspondence \eqref{eq:DivisorMapCorrespondence} (see \cite{CRThesis} for a nice exposition). 

\begin{theorem}[{\cite[Theorem 5.7]{SchTuckFiniteMaps}}]\label{thm:SchwedeTuckerTranspose}
  Let $R \subseteq S$ and $T : S \to R$ be as in Theorem \ref{thm:TransformationRuleF-Sig}, and let $f : \Spec{S} \to \Spec{R}$ be the corresponding map of schemes. Then $\phi : F_*^e R \to R$ has a transpose along $T$ if and only if $f^*\Delta_\phi - D_T$ is effective. If it exists, the transpose $\psi$ is unique and the divisor on $\Spec{S}$ associated to $\psi$ is $f^*\Delta_\phi - D_T$.
\end{theorem}

Now, we come to the crucial proposition which allows us to reduce to the Cartier case in the proof of the main theorem.

\begin{proposition}\label{prop:TransformationRuleSpecialCase}
  Let $(R,\Delta)$ be a strongly $F$-regular pair, where $R$ is an $F$-finite local ring and $\Delta \geq 0$. Let $D$ be a reduced, irreducible $\Q$-Cartier divisor such that $\O_{D}$ is strongly $F$-regular. For the extension $R \subseteq C(D)$ described above, with associated map of schemes $\pi : \Spec{C(D)} \to \Spec{R}$, we have the following.
  \begin{enumerate}
    \item $m\cdot s(R,\Delta) = s(C(D), \pi^*\Delta)$.
    \item If $D' = \pi^*D$, then $\O_{D'}$ is strongly $F$-regular.
    \item If the index of $(R, \Delta+D)$ is prime to $p$, then $m\cdot s(D,\Diff_D(\Delta)) = s(D', \Diff_{D'}(\pi^*\Delta))$.
  \end{enumerate}
\end{proposition}
\begin{proof}
  By Proposition \ref{prop:CyclicCoverProperties}, $C(D)$ is a normal local ring with maximal ideal $\mf{n} = \mf{m} \oplus \left( \bigoplus_{i=1}^{m-1} R(-iD)u^i \right)$. In particular, $\kappa(\mf{m}) = \kappa(\mf{n})$. Moreover, $[\Frac{S} : \Frac{R}] = m$ since each $R(iD)$ is a rank one reflexive sheaf. Hence the transformation rule (1) holds by Theorem \ref{thm:TransformationRuleF-Sig} and the fact that the projection to the degree 0 piece $T : C(D) \to R$ generates $\Hom_R(C(D),R)$ as a $C(D)$-module. Note that this statement does not require the strong $F$-regularity of $D$.
  
  For (2), $\O_{D'}$ is a $\Z/n\Z$-graded ring over $\O_D$ so is $S_2$ by \cite{TomWat}. So the final statement of Theorem \ref{thm:TransformationRuleF-Sig} implies that $s(\O_{D'}) > 0$ as long as $s(\O_D) > 0$. So $F$-regularity of $\O_{D'}$ follows from the $F$-regularity of $D$.
  
  For (3), since strongly $F$-regular rings are normal, we have the divisor-maps correspondence discussed in Section \ref{subsec:Different}. Let $\pi|_{D'} : D' \to D$ be the restriction of $\pi$ to $D'$. Let $\overline{T} : \O_{D'} \to \O_D$ be the restriction of the projection map $T : C(D) \to R$. Note that $\overline{T}$ is surjective and $\overline{T}(\mf{m}_{\O_{D}}) = \mf{m}_{\O_D}$, so it satisfies the hypotheses on the map in Theorem \ref{thm:TransformationRuleF-Sig}. Furthermore, the residue fields of $\O_D$ and $\O_{D'}$ agree by the explicit calculation of $\O_{D'}$ done in Proposition \ref{prop:CyclicCoverProperties}. So Theorem \ref{thm:TransformationRuleF-Sig} allows us to conclude once we show that $\Diff_{D'}(\pi^*\Delta) = \pi|_{D'}^*(\Diff_D(\Delta)) - D_{\overline{T}}$. To that end, consider the following diagrams defining $\Diff_D(\Delta)$ and $\Diff_{D'}(\pi^*\Delta)$. The vertical arrows are the natural quotient maps.
  \begin{center}
    \begin{tikzcd}
      F^e_*R \arrow[r, "\phi_{\Delta+D}"] \arrow[d] & R \arrow[d] & & F^e_*C(D) \arrow[r, "\phi_{\pi^*\Delta + D'}"] \arrow[d] & C(D) \arrow[d] \\
      F^e_*\O_D \arrow[r, "\phi_{\Diff_{D}(\Delta)}"] & \O_D & & F^e_*\O_{D'} \arrow[r, "\phi_{\Diff_{D'}(\pi^*\Delta)}"] & \O_{D'}
    \end{tikzcd}
  \end{center}
  We may connect these two diagrams with $T$ and $\overline{T}$ to form the following commutative cube

  \begin{center}
    \begin{tikzcd}
      & F^e_*C(D) \arrow[dl, "F_*^eT"'] \arrow[rrr, "\phi_{\pi^*\Delta + D'}"] \arrow[dd] & & & C(D) \arrow[dl, "T"] \arrow[dd] \\
      F^e_*R \arrow[rrr, crossing over, "\phi_{\Delta+D}" near end] \arrow[dd] & & & R \\
      & F^e_*\O_{D'} \arrow[dl, "F_*^e\overline{T}"'] \arrow[rrr, "\phi_{\Diff_{D'}(\pi^*\Delta)}"  near start] & & & \O_{D'} \arrow[dl, "\overline{T}"] \\
      F^e_*\O_{D} \arrow[rrr, "\phi_{\Diff_D(\Delta)}" ] & & & \O_{D} \arrow[from=uu, crossing over]\\
    \end{tikzcd}
  \end{center}
  The top of the cube commutes by Theorem \ref{thm:SchwedeTuckerTranspose} and the fact that $D_T = 0$. In particular, $\phi_{\Diff_{D'}(\pi^*\Delta)}$ lifts the map $\phi_{\Diff_D(\Delta)}$, and the statement follows from another application of Theorem \ref{thm:SchwedeTuckerTranspose}. 
\end{proof}

\begin{lemma}\label{lem:IndexPrimeToP}
  Let $R$ be a local ring, $D$ an integral $\Q$-Cartier divisor on $X = \Spec{R}$, and $\pi : Y = \Spec{C(D)} \to X$ be the associated cyclic cover. If $(X,\Delta)$ is a pair of index prime to $p$, then $(Y,\pi^*\Delta)$ has index prime to $p$.
\end{lemma}
\begin{proof}
  As before, let $T : C(D) \to R$ be the map which projects onto the degree zero piece of $C(D)$. We saw in the proof of Proposition \ref{prop:TransformationRuleSpecialCase} that $D_T = 0$, so we have the diagram
  \begin{center}
    \begin{tikzcd}
      F_*^e C(D) \arrow[r, "\phi_{\pi^*\Delta}"] \arrow[d, "F_*^eT"] & C(D)\arrow[d, "T"] \\
      F_*^eR \arrow[r, "\phi_\Delta"] & R 
    \end{tikzcd}.
  \end{center}
  By the divisor-maps correspondence discussed in Section \ref{subsec:Different}, we see that $(p^e-1)(K_Y + \pi^*\Delta)$ is Cartier, as required. 
\end{proof}

\begin{example}\label{ex:AnSingularities2}
  Let us carry out the cyclic cover construction for $R = \bb{F}_p\llbracket x,y,z \rrbracket / (xy - z^{n+1})$ and $D = V(x,z)$. Then $C = R \oplus (xt,zt) \oplus (xt^2, z^2t^2) \oplus \cdots \oplus (xt^{n}, z^nt^n)$ with $t^{n+1} = 1/x$. It is easy to see that $C / (xt^n) \cong \bb{F}_p\llbracket y,zt \rrbracket / (y-(zt)^{n+1})$, and the inclusion $R \subseteq C$ is given by $y \mapsto y$. We can also explicitly verify the equality of Proposition \ref{prop:CyclicCoverProperties}(4). We have $\Ram_{\pi|_{D'}} = n[0]$ where $[0]$ is the divisor corresponding to the point $(0,0)$ on $D'$ in the $(y,zt)$-plane. Also,
  \[
    \pi|_{D'}^*\Diff_D(0) = \pi|_{D'}^* \left( \frac{n}{n+1} \right)[0] = n[0]. 
  \]
  So $\pi|_{D'}^*\Diff_D(0) - \Ram_{\pi|_{D'}} = 0$. Since $D'$ is a smooth Cartier divisor, $C$ is Gorenstein. Hence $\Diff_{D'}(0)$ is easily seen to be 0 from the definitions. From these computations and Example \ref{ex:AnSingularities}, one sees that 
  \[
    (n+1)\cdot s(D,\Diff_D(\Delta)) = (n+1)\cdot s\left(\bb{F}_p\llbracket y\rrbracket, \left(\frac{n}{n+1}\right)[0]\right) = 1
  \]
  and
  \[
    s(D', \pi|_{D'}^*\Diff_D(\Delta) - \Ram_{\pi|_{D'}}) = s(\bb{F}_p\llbracket y \rrbracket) = 1
  \]
  which verifies part (2) of Proposition \ref{prop:TransformationRuleSpecialCase} in this case.
\end{example}

\section{Proof of the Main Theorem}\label{sec:InversionOfAdjunction}

This section contains the proof of the main theorem. In the proof, we employ two standard exact sequence computations whose proofs we include for completeness.

\begin{lemma}\label{lem:LengthLemma1}
  Let $R$ be a ring. If $I \subseteq R$ is an ideal of finite colength, $g \in R$ is a non-zero element, then 
  \[
    \ell_R\left( \frac{R}{(I : g)}\right) = \ell_R\left( \frac{R}{I} \right) - \ell_R\left( \frac{R}{(I,g)} \right).
  \]
\end{lemma}
\begin{proof}
  Take the kernel and cokernel of the map $R/I \map{\times g} R/I$ to obtain the exact sequence 
  \[
    0 \to \frac{(I:g)}{I} \to \frac{R}{I} \map{\times g} \frac{R}{I} \to \frac{R}{(I,g)} \to 0.
  \]
  So $\ell_R((I:g)/I) = \ell(R/(I,g))$. We also have $\ell((I:g)/I) = \ell(R/I) - \ell(R/(I:g))$. Equating the two expressions for $\ell((I:g)/I)$ gives the result. 
\end{proof}

\begin{lemma}\label{lem:LengthLemma2}
  Let $R$ be a ring, and let $M$ be an $R$-module of finite length. If $f \in R$ is a non-zero element, then
  \[
    \ell_R\left( \frac{M}{f^nM} \right) = n \cdot \ell_R\left( \frac{M}{fM} \right) - \sum_{i=1}^{n-1} \ell_R\left(\frac{(f^{i+1}M : f^i)}{fM} \right).
  \]
\end{lemma}
\begin{proof}
  For $n \geq 1$, we have the isomorphism $\frac{M}{(f^{n+1}M:f^n)} \cong \frac{f^nM}{f^{n+1}M}$ and exact sequences 
  \begin{align*}
    0 \to \frac{f^{n}M}{f^{n+1}M} \to& \frac{M}{f^{n+1}M} \to \frac{M}{f^{n}M} \to 0 \\
    0 \to \frac{(f^{n+1}M : f^{n})}{fM} \to &\frac{M}{fM} \to  \frac{M}{(f^{n+1}M : f^{n})} \to 0.
  \end{align*}
  The base case of $n=1$ being trivial, the result follows by induction on $n$. 
\end{proof}

Now, we have the tools to prove our main theorem. Recall that for a function $f(t)$, we denote the left derivative of $f$ at $c$ by $\frac{\partial_-}{\partial t}\big|_{t=c} f(t)$.

\begin{theorem}\label{thm:MainTheorem}
  Let $(R,\Delta + D)$ be a log $\Q$-Gorenstein pair. Assume $R$ is a normal local ring of dimension $d$, $\Delta$ is an effective $\Q$-divisor, and $D$ is a reduced, irreducible $\Q$-Cartier divisor such that $D \wedge \Delta = 0$. Then
  \[
     \frac{\partial_-}{\partial t}\bigg|_{t=1} s(R,\Delta + tD) = -s(\O_D,\Diff_D(\Delta)).
  \]
\end{theorem}
\begin{proof}
  First, we may assume that $(R, \Delta+D)$ has splitting dimension $d-1$. Otherwise, the $F$-pure threshold of the pair is strictly less than one, and the result is trivial. Thus, we may assume that $D$ is the minimal center of $F$-purity for $(R,\Delta+D)$, so is $F$-regular and hence normal \cite{SchFad}.
  
  Now, we handle the case where $D = \div(f)$ is Cartier and $K_R + \Delta + D$ has index prime to $p$. The limit defining $\frac{\partial_-}{\partial t}\big|_{t=1} s(R,\Delta + tD)$ exists by Proposition \ref{prop:ContinuityAndConvexityOfF-signature}, so we may compute the left derivative via any sequence converging to $1$ from below. We choose the sequence $1 - 1/p^e$. We have
  \begin{align*}
    \frac{\partial_-}{\partial t}\bigg|_{t=1} s(R,\Delta + tD) &= \lim_{t \to 1^-} \frac{s(R,\Delta +tD) - s(R,\Delta +D)}{t-1} \\ 
    &= \lim_{e \to \infty} \frac{s(R,\Delta +(1-1/p^e)D) - s(R,\Delta+D)}{-1/p^e} \\
    & = -\lim_{e \to \infty} \lim_{r \to \infty}  \frac{1}{p^{rd-e}}\left( \ell_R\left( \frac{R }{ (I^\Delta_r:f^{p^{r} - p^{r-e}})} \right) - \ell_R\left(\frac{R}{(I^\Delta_r : f^{p^r})}\right) \right)
  \end{align*}
  Applying Lemmas \ref{lem:LengthLemma1} and \ref{lem:LengthLemma2} gives
  \begin{align*}
    \frac{\partial_-}{\partial t}\bigg|_{t=1} s(R,\Delta + tD) &= -\lim_{e \to \infty} \lim_{r \to \infty} \frac{1}{p^{rd-e}} \left( \ell_R\left(\frac{R}{(I^\Delta_r,f^{p^r})}\right) - \ell_R\left(\frac{R }{ (I^\Delta_r,f^{p^r - p^{r-e}})}\right) \right) \\
    &= -\lim_{e \to \infty} \lim_{r \to \infty} \frac{1}{p^{rd-e}} \left( p^{r-e}\ell_R\left(\frac{R}{(f,I^\Delta_r)}\right) - \sum_{i=p^r - p^{r-e}}^{p^r-1} \ell_R\left( \frac{(\ov{f}^{i+1} : \ov{f}^i)}{\ov{f}} \right)\right)
  \end{align*}
  where $\ov{f}$ denotes the image of $f$ in $R/I_r^\Delta$. Note that the ideals $(\ov{f}^{i+1}:\ov{f}^i)$ form an increasing chain
  \[
    (\ov{f}^2 : \ov{f}) \subseteq (\ov{f}^3 : \ov{f}^2) \subseteq \cdots \subseteq (\ov{f}^{p^r} : \ov{f}^{p^r-1}) = (0 : \ov{f}^{p^r-1})
  \]
  In particular, $\ell_R\left( \frac{(\ov{f}^{i+1}:\ov{f}^i}{(\ov{f})} \right) \leq \ell_R\left( \frac{(0:\ov{f}^{p^r-1})}{(\ov{f})}\right)$ for $i \leq p^r-1$. So we have the inequality
  \begin{align*}
    \frac{\partial_-}{\partial t}\bigg|_{t=1} s(R,\Delta + tD) &\leq -\lim_{e \to \infty} \lim_{r \to \infty}\frac{1}{p^{r(d-1)}} \left( \ell\left( \frac{R}{(I^\Delta_r,f)}\right) - \ell\left( \frac{(0:\ov{f}^{p^r-1})}{(\ov{f})} \right) \right) \\
    &=  -\lim_{r \to \infty} \frac{1}{p^{r(d-1)}} \left( \ell\left(\frac{R}{(I^\Delta_r,f)}\right) -  \ell\left(\frac{R}{(I^\Delta_r,f)}\right) + \ell\left(\frac{R}{(I^\Delta_r:f^{p^r-1})} \right) \right) \\
    &= - \lim_{r \to \infty} \frac{1}{p^{r(d-1)}}\ell\left( \frac{R}{(I^\Delta_r : f^{p^r-1})} \right) \\
    &= - s(\O_D, \Diff_D(\Delta))
  \end{align*}
  where the final equality follows from Lemma \ref{lem:LimitLemma}.
  
  For the reverse inequality, the inclusion $(I^\Delta_e:f^{p^e-1})^{[p^{r-e}]} \subseteq (I^\Delta_r: f^{p^r - p^{r-e}})$ of Lemma \ref{lemma:ColonFrobeniusPowerContainment} implies
  \begin{align*}
    \frac{\partial_-}{\partial t}\bigg|_{t=1} s(R,\Delta + tD) &= \lim_{e \to \infty} \frac{s(R,f^{1-1/p^e})}{-1/p^e} \\
    &= - \lim_{e \to \infty} \lim_{r \to \infty} \frac{1}{p^{rd-e}} \ell_R\left(  \frac{R}{(I^\Delta_r : f^{p^r - p^{r-e}})}\right) \\
    &\geq - \lim_{e \to \infty} \lim_{r \to \infty} \frac{1}{p^{rd-e}} \ell_R\left(\frac{ R}{(I^\Delta_e : f^{p^e-1})^{[p^{r-e}]}} \right) \\
    &= -\lim_{e \to \infty} \frac{\HK((I^\Delta_e:f^{p^e-1});R)}{p^{e(d-1)}}. \\
  \end{align*}
  By Proposition \ref{prop:CartierSubalgebraComparison}, the extension of $(I^\Delta_e:f^{p^e-1})$ in $R/f$ is $I_e^{\Diff_D(\Delta)}$ for large and divisible $e$. Thus 
  \[
    \frac{\partial_-}{\partial t}\bigg|_{t=1} s(R,\Delta + tD) \geq - \lim_{e \to \infty} \frac{\HK(I_e^{\Diff_D(\Delta)};\O_D)}{p^{e(d-1)}} = -s(\O_D, \Diff_D(\Delta))
  \]
  where the last equality follows from Theorem \ref{thm:F-signatureAndHK} and Proposition \ref{prop:CartierSubalgebraComparison}.

  Now, we use the Cartier case and the results of Section \ref{sec:VeroneseCovers} to handle the case where $D$ is $\Q$-Cartier of index $m$. So $mD = \div(x)$ for some $x \in R$. Let $\mf{p}$ be the ideal cutting out the $\Q$-Cartier divisor $D$. Form the finite ring map $R \to C(D) := R \oplus \mf{p}t \oplus \mf{p}^{(2)}t^2 \oplus \cdots \oplus \mf{p}^{(m-1)}t^{m-1}$ with $\pi : \Spec{C(D)} \to \Spec{R}$ the corresponding map of schemes. By Proposition \ref{prop:CyclicCoverProperties}, $D' := \pi^*D$ is Cartier. By Proposition \ref{prop:TransformationRuleSpecialCase}, $D'$ is also $F$-regular, so in particular, $D'$ is normal. Moreover, the pair $(C(D), \pi^*\Delta + D')$ has index prime to $p$ by Lemma \ref{lem:IndexPrimeToP}. We compute
  \[
    \frac{\partial_-}{\partial t}\bigg|_{t=1} s(R,\Delta + tD) = \lim_{e \to \infty} \frac{s(R,\Delta+(1-1/p^e)D)}{-1/p^e} = \lim_{e\to \infty} \frac{s(C(D),\pi^*\Delta + (1-1/p^e) D')}{-m/p^e}  
  \]
  \[
    =\frac{-s(D', \Diff_{D'}(\pi^*\Delta))}{m} = -s(D,\Diff_D(\Delta)).
  \]
  The first equality is part (1) of Proposition \ref{prop:TransformationRuleSpecialCase}, the second follows from the Cartier case, and the third follows from part (2) of Proposition \ref{prop:TransformationRuleSpecialCase}.

  Finally, we remove the assumption that the $\Q$-Gorenstein index is prime to $p$ via a perturbation of $\Delta$. See \cite[Corollary 5.4]{Das} for another instance of this perturbation argument. Assume $K_R + \Delta + D$ is $\Q$-Gorenstein of arbitrary index. Choose $B \geq 0$ such that $K_R - B$ is Cartier, and choose an effective $D' \sim D$ such that $D'$ does not contain $D$. Then set $\Delta' = D' + \Delta + B$. Thus $\Delta'$ is an effective $\Q$-Cartier divisor such that $K_R + (\Delta + \frac{1}{p^e-1} \Delta') + D$ is $\Q$-Cartier with index prime to $p$.
  
  Now, let $s_e(t) = s(R, \Delta + \frac{1}{p^e-1}\Delta' + tD)$. Note that $s_e(t)$ is a sequence of convex continuous functions of $t \in [0,1]$ which converges pointwise to $s(t) := s(R, \Delta + tD)$ and strictly increases with $e$. Thus $s_e(t)$ converges uniformly to $s(t)$. Since $\Delta'$ and $\Diff_{D}(\Delta')$ $\Q$-Cartier, Proposition \ref{prop:ContinuityAndConvexityOfF-signature} implies 
  \begin{align*}
  	\lim_{e \to \infty} s\left(R, \Delta + \frac{1}{p^e-1}\Delta' + tD\right) &= s(R, \Delta + tD) \\ 
  	\lim_{e \to \infty} s\left(\O_D, \Diff_D\left(\Delta + \frac{1}{p^e-1}\Delta'\right)\right) &= s(\O_D, \Diff_D(\Delta))
  \end{align*}
  for each $t$. Applying the previous case of the theorem at hand and uniform convergence to switch the limit and the derivative, we have
	\begin{align*}
		\frac{\partial_-}{\partial t}\bigg|_{t=1} s(R, \Delta + tD) &= \frac{\partial_-}{\partial t}\bigg|_{t=1} \lim_{e \to \infty} s\left(R, \Delta + \frac{1}{p^e-1}\Delta' + tD \right) \\
		&= \lim_{e \to \infty} \frac{\partial_-}{\partial t}\bigg|_{t=1} s\left(R, \Delta + \frac{1}{p^e-1}\Delta' + tD \right) \\
		&= \lim_{e \to \infty} -s\left(\O_D, \Diff_D\left(\Delta + \frac{1}{p^e-1}\Delta' \right) \right) \\
		&= -s(\O_D, \Diff_D(\Delta))
	\end{align*}
  	as required.
\end{proof}

\begin{remark}\label{rmk:CyclicCoverIsNecessary}
  One might wonder whether the cyclic cover construction is necessary for the proof of Theorem \ref{thm:MainTheorem}. The proof in the Cartier case can be carried out in the $\Q$-Cartier case, using Lemma \ref{lem:LimitLemma} in its most general form. However, this approach only yields the bound $-s(D,\Diff_D(\Delta)) \leq \frac{\partial_-}{\partial t}\big|_{t=1} s(R,\Delta + tD) \leq -s(D,\Diff_D(\Delta)) / m$ where $m$ is the index of $D$.
\end{remark}

\begin{remark}\label{rmk:QCartierDivisor}
	The inversion of adjunction statement in \cite{Das} does not require the divisor $D$ to be $\Q$-Cartier. It would be interesting to know whether Theorem \ref{thm:MainTheorem} holds in this more general case. To the knowledge of the author, the corresponding result for normalized volume in characteristic 0 (i.e. Theorem \ref{thm:LiLiuXuTheorem}) is also open in the case where $D$ fails to be $\Q$-Cartier.
\end{remark}

\section{Complements}\label{sec:Corollaries}

\subsection{Inversion of Adjunction for the $F$-Signature}\label{subsec:Inequalities}

As noted previously, one can think of Theorem \ref{thm:MainTheorem} as a quantitative version of inversion of adjunction. Indeed, the theorem immediately implies the following inequality.

\begin{corollary}[Inversion of Adjunction for $F$-Signature]\label{cor:InversionOfAdjunctionFSignature}
  With notation as in Theorem \ref{thm:MainTheorem}, we have
  \[
    s(R,\Delta) \geq s(D,\Diff_D(\Delta))
  \]
  with equality if and only if the function $s(R,\Delta+tD)$ is linear for $t \in [0,1]$ with slope $-s(R,\Delta)$. In particular, if $(D,\Diff_D(\Delta))$ is strongly $F$-regular, so is $s(R,\Delta)$.
\end{corollary}
\begin{proof}
  By Proposition \ref{prop:ContinuityAndConvexityOfF-signature}, the $F$-signature function $s(R,\Delta+tD)$ is convex. Thus 
  \[
    \frac{\partial_-}{\partial t}\bigg|_{t=1} -s(R,\Delta + tD) \leq s(R,\Delta),
  \] 
  and the statement follows from the inequality 
  \[
    s(D,\Diff_D(\Delta)) \leq -\frac{\partial_-}{\partial t}\bigg|_{t=1} s(R,\Delta + tD)
  \] 
  of Theorem \ref{thm:MainTheorem}. The characterization of equality is an immediate consequence of the convexity of the function $s(R,\Delta+tD)$. 
\end{proof}

The final statement in Corollary \ref{cor:InversionOfAdjunctionFSignature} is known as \emph{strongly $F$-regular inversion of adjunction} and was originally shown in \cite[Corollary 5.4]{Das}. Note that Das's result does not require $D$ to be $\Q$-Cartier. As stated in the introduction, inversion of adjunction is a key tool in the MMP. For example, in \cite{HacWit}, Das's inversion of adjunction is used to establish the existence of generalized plt blowups in dimension 3. It would be interesting to further characterize when equality holds in Corollary \ref{cor:InversionOfAdjunctionFSignature} in terms of the geometry of $R$ and $D$.

The question of inversion of adjunction is closely related to the notion of deformation for singularities. We say that a property $\mcal{P}$ \emph{deforms} if the following implication holds: 
\[
	R/x \text{ satisfies } \mcal{P} \text{ for some regular element } x \Longrightarrow R \text{ satisfies } \mcal{P}.
\]
Understanding deformation of $F$-singularities has attracted great interest over the last twenty years. Deformation of $F$-purity fails in general, but was recently shown to hold for $\Q$-Gorenstein rings in \cite{SimpPol}. Deformation of $F$-rationality is known, but deformation of $F$-injectivity is one of the eminent open problems of positive characteristic commutative algebra. It is well-known that $F$-regularity does not deform in general but does deform under mild conditions \cite{Singh99a}. In fact, Corollary \ref{cor:InversionOfAdjunctionFSignature} recovers and refines a special case of the main result in \cite{Singh99a}. Indeed, suppose $R$ is a normal, $\Q$-Gorenstein local ring and $D = \Spec{R/f}$ is $F$-regular. Then $\Diff_D(0) = 0$, and Corollary \ref{cor:InversionOfAdjunctionFSignature} implies $s(R) \geq s(R/f)$. Since positivity of the $F$-signature characterizes $F$-regularity, $R$ is $F$-regular as well. The inequality $s(R) \geq s(R/f)$ can be interpreted as \emph{deformation of the $F$-signature}. Given our main theorem and the fact that $F$-rationality deforms, it would be interesting to study the behavior of the various candidates for $F$-rational signature \cite{SmirTuck} under deformation and inversion of adjunction.

\begin{example}\label{ex:AnSingularities3}
 Let $R = \bb{F}_p\llbracket x,y,z\rrbracket/(xy-z^{n+1})$ and $D = V(x,z)$. Then Corollary \ref{cor:InversionOfAdjunctionFSignature} shows that $s(R, tD) = \frac{1-t}{n+1}$ for all $t \in [0,1]$. Indeed, $s(R) = 1/(n+1)$ by \cite[Example 18]{HunLeu} and as computed in Example \ref{ex:AnSingularities}, $s(D, \Diff_D(0)) = s(\bb{F}_p\llbracket y \rrbracket, \left( \frac{n}{n+1} \right)[0]) = \frac{1}{n+1}$. 
 
 One should note that the $A_n$ singularities, used as a running example in this paper, are toric. So the $F$-signature function can also be computed with toric methods \cite{VK}.
\end{example}

\subsection{$F$-Signature and Normalized Volume}\label{subsec:NormalizedVolume}
To conclude the paper, we take a moment to expand on the connection of our results with the study of normalized volume. The normalized volume, introduced in \cite{Li18}, is a numerical invariant used to study klt pairs in characteristic 0. Since its introduction, it has inspired a great deal of work, especially due to its implications for $K$-semistability and moduli of Fano varieties, see \cite{LiLiuXu} and the references therein. Although no precise conjectures have been formulated, there is mounting evidence that the $F$-signature has deep connections to normalized volume. In particular, the normalized volume and $F$-signature have similar behavior under small birational maps \cite{MPST}, and in forthcoming work Yuchen Liu defines a positive characteristic analog of the normalized volume, called the $F$-volume, for which he proved inequalities relating the $F$-volume to the $F$-signature \cite{LiLiuXu}. Most relevant to this paper is the following result computing the left derivative of the normalized volume function.

\begin{theorem}[{\cite[Proposition 6.8]{LiLiuXu}}]\label{thm:LiLiuXuTheorem}
  Let $x \in (X,\Delta)$ be an $n$-dimensional klt singularity. Let $D$ be a normal $\Q$-Cartier divisor containing $x$ such that $(X,D+\Delta)$ is plt. Denote by $\Diff_D(\Delta)$ the different of $\Delta$ on $D$. Then
  \[
    \lim_{t \to 1^-} \frac{\nvol(x, X, \Delta + tD)}{-n^n t} = -\frac{\nvol(x,D,\Diff_D(\Delta))}{(n-1)^{n-1}}.
  \]
\end{theorem}

Theorem \ref{thm:LiLiuXuTheorem} was the inspiration for our main theorem. These results demonstrate that the behavior under adjunction at codimension 1 centers of both the $F$-signature and the normalized volume is controlled by the left-derivative at $t = 1$. Theorem \ref{thm:MainTheorem} provides further evidence of the connection between these two invariants.

It is natural to ask about adjunction at centers of higher codimension. In their survey, Li, Liu, and Xu ask whether a version of Theorem \ref{thm:LiLiuXuTheorem} holds for centers of higher codimension \cite[Question 6.10]{LiLiuXu}. Using Schwede's theory of $F$-adjunction, there is a notion of the $F$-different for centers of higher codimension. Although, one should note that the different and the $F$-different do not necessarily agree for centers of higher codimension \cite{DasSch}. 

\begin{question}[{cf. \cite[Question 6.10]{LiLiuXu}}]\label{ques:HigherCodimension}
	Let $(R,\Delta)$ be a normal $\Q$-Gorenstein (local) pair of dimension $d$ and let $D$ be an effective $\Q$-Cartier divisor. Let $c = \mathrm{fpt}(R, \Delta; D)$ be the $F$-pure threshold, let $W$ be the minimal center of $F$-purity, and let $\FDiff_W(\Delta)$ be the $F$-different \cite{SchFad}. Suppose that $W$ has codimension $k$ in $\Spec{R}$. Can we bound the limit
	\[
		\lim_{\varepsilon \to 0^+} \frac{s(R,\Delta + (1-\varepsilon )cD)}{\varepsilon^k}
	\]
	from below in terms of $s(R,\Delta)$ and $s(W, \FDiff_W(\Delta))$?
\end{question}

For this question, it might be necessary to assume that $p$ does not divide the denominator of the $F$-pure threshold, as this avoids some pathological behavior. For regular rings, the main result of \cite{CantonEtAl} suggests that we have the right power of $\varepsilon$ in the denominator of the limit in Question \ref{ques:HigherCodimension}.

\bibliographystyle{alpha}
\bibliography{Fadjunction}{}

\end{document}